\documentclass[11pt]{amsart}

\usepackage{amsmath,amsthm,amsfonts,amssymb,amscd,mathrsfs}
\usepackage[margin=3cm]{geometry}
\usepackage{color}
\usepackage{enumitem}

\usepackage{xcolor,cancel}

\usepackage{ulem}
\usepackage{bbm}

\usepackage{tikz}
\usetikzlibrary{calc}
\usetikzlibrary{arrows,shapes,chains}

\usepackage{comment}
\usepackage{graphicx, graphics}
\usepackage[pdftex]{pict2e}
\usepackage{cases}
\usepackage{todonotes}
\usepackage{psfrag,pst-node,subfigure,rotating,
amsmath, bbm, amsthm, amssymb, amsthm, setspace, picture, epsfig,
amsfonts}
\usepackage{multirow}
\usepackage{indentfirst}
\usepackage[colorlinks,linkcolor=blue,anchorc
olor=green,citecolor=blue]{hyperref}
\numberwithin{equation}{section}
\newcommand{\bea}{\begin{eqnarray}}
\newcommand{\eea}{\end{eqnarray}}
\newcommand{\Bea}{\begin{eqnarray*}}
\newcommand{\Eea}{\end{eqnarray*}}

\theoremstyle{plain}

\newtheorem{Thm}{Theorem}[section]
\newtheorem{Lem}[Thm]{Lemma}

\newtheorem{Prop}[Thm]{Proposition}

\theoremstyle{remark}

\long\def\begcom#1\endcom{}

\newcommand{\Sep}{\operatorname{Sep}}
\newcommand{\wSep}{\widehat{{\rm Sep}}}

\newcommand{\rd}{{\rm{d}}}
\newcounter{RomanNumber}

\def \N{{\mathbb N}}

\def\DS{\displaystyle}

\def\ln{\operatorname{ln}}

\def\one{{\mathbf{1}}}

\def\bbD{{\mathbb{D}}}
\def\bbH{{\mathbb{H}}}

\def\NN{\mathbb{N}}

\def\brS{{\bar{S}}}
\def\brbrS{{\bar{\brS}}}
\def\brX{{\bar{X}}}
\def\brbrX{{\bar{\brX}}}

\def\breps{{\bar\varepsilon}}

\def\bS{{\mathbf{S}}}

\def\cE{{\mathcal E}}

\def\NN{{\mathbb N}}

\def\tW{{\widetilde W}}
\def\tS{{\tilde S}}
\def\bralpha{{\bar \alpha}}
\def\brn{{\bar n}}
\def\hn{{\hat n}}
\def\Var{{\rm Var}}
\def\brj{{\bar j}}
\def\hA{{\hat A}}
\def\hp{{\hat p}}

\def\fA{\Omega}

\def\fr{{\rho}}

\def\fv{\sigma}

\def\hB{{\hat{B}}}
\definecolor{bluegray}{rgb}{0.1, 0.1, 0.6}

\definecolor{dgreen}{rgb}{0.1,0.6,0.1}

\definecolor{bluegreen}{rgb}{0.1,0.5,0.2}
\definecolor{bpurple}{rgb}{0.74,0.2,0.64}

\definecolor{orange}{rgb}{0.8, 0.33, 0.0}

\def\EXP{{\mathbb{E}}}

\def\eps{{\varepsilon}}
\definecolor{deepcarrotorange}{rgb}{1, 0.31, 0.0}

\title[An analogue of LIL for heavy tailed random variables]{An analogue of Law of Iterated Logarithm for Heavy Tailed Random Variables.}

\author{Dmitry Dolgopyat}
\address[D.~Dolgopyat]{Department of Mathematics\\
University of Maryland, 4417 Mathematics Bldg, College Park\\
MD 20742, USA}
\email{dolgop@umd.edu}
\author{Sixu Liu}
\address[S.~Liu]{Beijing Institute of Mathematical Sciences and Applications\\
 Beijing 101408, China.
 }
\email{liusixu@bimsa.cn}

\begin{document}
\begin{abstract}
We establish functional limit theorems for ergodic sums of observables with power singularities for expanding circle maps. In the regime where the observables have infinite variance, we show that when rescaled by $N^{1/s} (\ln N)^\alpha$, the partial sum process has limit points consisting precisely of increasing piecewise constant functions with finitely many jumps. Our approach combines trimming techniques with a multiple Borel-Cantelli argument.
It provides a functional law of the iterated logarithm for heavy-tailed
processes where classical almost sure invariance principles do not apply.
\end{abstract}

\maketitle
\section{Introduction}
Let $\{X_k\}_{k=1}^\infty$ be a sequence of independent and identically distributed random variables with mean 0 and variance 1. Define the partial sum $S_n = \sum_{k=1}^n X_k$ and consider the function $W_N: [0,1] \to \mathbb{R}$ given by $$W_N\left(\frac{n}{N}\right)=\frac{S_n}{\sqrt{N}}$$ at rational points $n/N,$ extended by linear interpolation elsewhere. The functional central limit theorem states that  $W_N(\cdot)$  converges in distribution  to the standard Wiener measure as $N\to\infty$. Since Wiener measure has full support on $C([0,1])$, it follows that almost surely, for any continuous $h \in C([0,1])$, there exists a subsequence $N_j$ such that $W_{N_j}$ converges uniformly to $h$ (see e.g. \cite{R76}).

It is natural to ask whether one can obtain a smaller set of limits by considering a larger normalization.
Strassen \cite{Str64} answered this question by defining
$$\label{FLIL}
 \tW_N\left(\frac{n}{N}\right)=\frac{S_n}{\sqrt{2N\ln \ln N}},
$$
and proving that almost surely, the limit points of $\{\widetilde{W}_N\}$ are precisely the absolutely continuous functions $h$ with $h(0)=0$ and $\int_0^1 |h'(x)|^2\rd x\leq 1.$  This result is referred to as the functional form of law of the iterated logarithm~(FLIL) in the literature~\cite{HH80,PS75}. FLIL  generalizes the classical LIL for random variables~\cite{HW41,Khintchine24,Kolmogoroff24}.

Strassen's proof relies on the almost sure invariance principle (ASIP), which establishes that, under suitable conditions, a martingale is almost surely close to a Brownian motion. Consequently, FLIL extends to martingales satisfying appropriate moment conditions \cite{HH80,Str67}. The ASIP framework further extends to dynamical systems via martingale approximations for weakly dependent processes \cite{PS75}, establishing FLIL for systems including uniformly hyperbolic systems~\cite{DP84}, Lorenz attractors~\cite{HM07}, non-uniformly hyperbolic systems~\cite{Haydn15,Korepanov18, MN05, MN09}, partially hyperbolic systems~\cite{Dolgopyat04}, dispersing billiards~\cite{BM08,Chernov06} and compact group extensions~\cite{FMT03}.

However, there are still many examples of dynamical systems where the limit distributions of ergodic sums are non-Gaussian. For these systems, the aforementioned approach is not applicable, and ASIP does not hold. It is therefore of interest to develop a new method for determining a suitable normalizing sequence $L_N$ such that the function
$$\tS_N\left(\frac{n}{N}\right)=\frac{S_{n}}{L_N}$$
has a non-trivial set of limit points.

In this paper, we address the above problem in a simple setting. Let
$$\DS S_n(y)=\sum_{k=1}^n \phi(f^k y),$$ where $f$ is an expanding circle map  $\mathbb{S},$ and
$\phi$ has a power singularity of the form $|y-x|^{-1/s}$ for a fixed point $x\in\mathbb{S}.$ Assuming that $0<s<2$ and $x$ is a Lebesgue typical point, we show that a natural normalizing sequence is
$$L_N=N^{1/s} (\ln N)^\alpha.$$ If $\alpha$ satisfies
$${[(r+1)s]}^{-1}<\alpha\leq{(rs)}^{-1},$$ then the set of limit points consists of piecewise constant increasing  functions starting at $0$ with at most $r$ jumps.

Before formulating our result precisely, let us mention two related topics.

The first topic is laws of large numbers for trimmed sums (see \cite{AN, KS19a, KS20, KM95, Mo76}). Trimming is an important tool in our analysis, as it allows us to discard all summands in $S_N$ except for the $r$ largest ones. However, this is only part of our analysis, as the contribution of the retained terms requires more detailed analysis than in the papers mentioned above. Additionally, unlike the studies cited above, our results apply not only to sums with infinite mean but also to sums with infinite variance.

The second topic involves the application of the classical Borel-Cantelli lemma to estimate the rate of convergence
in laws of large numbers (see \cite{CN, Ch66, DGM, Kes97} and references wherein). This topic is closer to
the subject of our paper. However, while the papers mentioned above deal with a single  $N$ at a time, we obtain
the functional version of the deviation result. Fundamental to our approach is the multiple Borel-Cantelli lemma proposed recently in \cite{DBS21}, which provides a criterion for determining when multiple rare events occur at the same time scale.

The paper is structured as follows. Section \ref{ScResults} states the main theorem and outlines its proof via the multiple Borel-Cantelli lemma. Section \ref{ScHeavy} proves a key lemma controlling trimmed sums using reflection principles. Standard technical results are deferred to appendices.

\section{An analogue of LIL for Smooth expanding maps}
\label{ScResults}

\subsection{The main result.}
Let $f$ be a $C^2$ expanding map of the circle $\mathbb{S}.$ Here $f$ is called expanding if there exists $\gamma>1$ such that $|f'(x)|>\gamma$ for all $x\in\mathbb{S}.$   It is known that $f$ admits a unique invariant measure $\mu$ that is absolutely continuous with respect to the Lebesgue measure~\cite{Viana97}. Fix $x\in\mathbb{S},$ and consider an unbounded observable $\phi$ of the form
$$\phi(y)=\frac{a}{|y-x|^{1/s}}+\psi(y),$$ where $\psi$ is a $\psi \in C^1(\mathbb{S}),$ $a>0,$ and $s>0.$
Define the sequence
$\DS X_k(y) =\phi(f^k y).$ Then
\begin{equation}
\label{Tail}
\mu\left(\left\{y:X_k(y)>t\right\}\right)\sim a_st^{-s}
\end{equation}
for sufficiently large $t,$ where $a_s=\rho_0 a^{s}$ and $\rho_0$ denotes the density of $\mu$ at $x.$
Now for a positive parameter $\alpha,$ define the process
\bea\label{WN} W_N(t)=\frac{\left(\sum_{k\leq Nt} X_k\right)-a_N}{N^{1/s} (\ln N)^\alpha},
\eea
with centering term $a_N=0$ if $s\leq1,$ and $a_N=Nt\int \phi(y) \rd\mu(y)$ if $s>1.$

We regard $W_N(\cdot)$ as a sequence of random elements in the space $\bbD[0,1]$ of c\`{a}dl\`{a}g functions from $[0,1]$ to
$\overline{\mathbb{R}}=[-\infty, +\infty],$ equipped with the metric $$d(z_1, z_2)=\left|\tan^{-1}(z_1)-\tan^{-1}(z_2)\right|.$$
While several topologies exist on $\bbD[0,1]$, the Skorokhod $J_1$ topology is most relevant for our analysis.
Let $\Lambda$ denote the set of strictly increasing continuous functions $\lambda:[0,1] \to [0,1]$ with $\lambda(0) = 0$ and $\lambda(1) = 1$.  The $J_1$ topology is defined \cite{Kern23} using the distance
$$d_{J_1}(h_1, h_2)=\inf_{\lambda\in\Lambda} \left\{\|h_1\circ\lambda-h_2\|_\infty+\|\lambda-id\|_\infty\right\},$$
where $\|\cdot\|_\infty$ denotes the uniform norm.

Key properties of the $J_1$ topology that are employed in our analysis include:
\begin{enumerate}[leftmargin=*]
    \item[(i)] Convergence in uniform norm implies convergence in $J_1$ topology.
    \item[(ii)] Let $\bbH_r$ be the space of step functions with at most $r$ positive jumps, then $W\in\bbH_r$  can be written as
    $$ W(t)=\sum_{m=1}^{\bar r} c_m \one_{t\leq t_m}(t), $$
    where $\bar{r} \leq r$, the coefficients $\{c_m\}_{m=1}^{\bar{r}}$ represent positive amplitudes, and the points $\{t_m\}_{m=1}^{\bar{r}}$ are strictly increasing jump times. Two elements $W_1, W_2 \in \bbH_r$ with exactly $r$ jumps are $J_1$-close if their respective amplitudes and jump times are close.
    \item[(iii)] The $J_1$ topology allows us to ignore small jumps in the limit.
\end{enumerate}

To present our results, we introduce the notion of Diophantine points.
A point $x$ is \textbf{Diophantine} if there exist constants $\epsilon=\epsilon(x)>0$ and $\rho_0=\rho_0(x)>0$ such that for all  $\rho\leq \rho_0$ and all positive integers $k\leq \epsilon|\ln \rho|,$
$$f^{-k} B(x, \rho) \cap B(x, \rho)
= \emptyset.$$
This notion of Diophantine points characterizes a specific type of slow recurrence behavior. Rooted in Diophantine approximation, which restricts how well real numbers can be approximated by rational ones, this property effectively prevents dense orbit clustering in dynamical systems. Unlike periodic points, Diophantine points are quantitatively
non-periodic. Furthermore, as stated in Lemma~5 of \cite{BarreiraSaussol01}, $\mu$-almost every $x$ is Diophantine.

With this definition established, we now proceed to formally state our main theorem.

\begin{Thm}
  \label{ThGMLIL}
Under the above hypotheses, assume that either

  (i) $0<s<1,$ or

  (ii) $1\leq s<2$  and $x$ is Diophantine.

  Let $\alpha$ satisfy
  $$\frac{1}{(r+1) s}<\alpha\leq \frac{1}{rs}$$ for some $r\in \NN.$
Then for almost every $y,$ the following holds: the set of
$J_1$ limit points of $\{W_N(t)\}_{N=1}^\infty$
is equal to $\bbH_r.$
\end{Thm}

\subsection{Multiple Borel-Cantelli lemma.} Given sequences of positive numbers $\{\rho_N\}$ and $\{c_j\},$ and some $\varepsilon<c_j.$ The proof of Theorem~\ref{ThGMLIL}
requires controlling the number of times $k\leq N$ such that the following events occur:
$$\left\{y:X_k(y)> \rho_N\right\}$$
and
$$\left\{y:\,c_j-\varepsilon\leq X_k(y)\leq c_j+\varepsilon\right\}.$$
Furthermore, if multiple such events occur for some $k,$ controlling their timing is essential. To address this,we employ an extension of the multiple Borel–Cantelli lemma, generalizing Theorem 2.12 in \cite{DBS21}.

Throughout,  we denotes $C$ a constant depending on the fixed number $r$ of the occurring events, which may vary across lines but remains independent of the time scale $N,$ the events $\fA_{\rho_N},$ and the iteration order of $f,$ etc.

Consider a probability space $(\Omega,\mathcal{F},\mu)$ and  a family of events $\{\fA_{\rho_N} \}_{N=1}^\infty.$
Suppose there exists a constant $\hat\varepsilon>0$ and a decomposition of $\fA_{\rho_N}$ such that
$$\fA_{\rho_N}=\bigcup_{i=1}^p \fA_{\rho_N,i}\quad\text{and}\quad
\mu(\fA_{\rho_N,i})\geq \hat\varepsilon\mu(\fA_{\rho_N}).$$
We introduce the following notation:
\begin{equation}\label{notation}
\fA^k_{\rho_N}=f^{-k}(\fA_{\rho_N}),\,\,\,\fA^k_{\rho_N,i}=f^{-k}(\fA_{\rho_N,i}),\,\,\,\sigma_{\rho_N}=\mu(\fA_{\rho_N})\,\,\, {\text{and}}\,\,\, \sigma_{\rho_N,i}=\mu(\fA_{\rho_N,i}).
\end{equation}
For a fixed $r\in\mathbb{N^*},$ and  each tuple $(i_1,\dots, i_r)\in \{1,\dots, p\}^r,$  we characterize whether, with probability 1,  the events $\left\{\fA_{\rho_N,i_j}^{k_j}\right\}$ occur for all $1\leq j\leq r$ for infinitely many $N.$

For an  $r$-tuple $0< k_1<\dots<k_r \leq N,$ we take the separation indices
\begin{align*}\Sep_N(k_1,\dots, k_r)=\sharp\left\{j:0\leq j\leq r-1,\, k_{j+1}-k_j\geq  s(N)\right\}, \,\,\,  k_0:=0,\\
  \wSep_N(k_1,\dots, k_r)=\sharp\left\{j:0\leq j\leq r-1,\, k_{j+1}-k_j\geq \hat s(N)\right\}, \,\,\, k_0:=0,\end{align*}
where $s: \mathbb{N} \to \mathbb{N}$ satisfies $$s(N) \leq C(\ln N)^2,$$ and  $\hat{s}:\N \rightarrow\N$ satisfies $$\varepsilon N\leq \hat s(N) <qN/{(2r)}$$ for some constants $0<q<1$ and  $0<\varepsilon<q/{(2r)}.$

Assume there exists a sequence $\eps_N \to 0$ satisfying the following conditions, which quantify mixing and non-clustering for  $\left\{\fA^{k}_{\fr_N}\right\}$ and  $\left\{\fA^{k}_{\fr_N, i}\right\}$ under appropriate separations.

\begin{itemize}
\item[$(M1)_r$] For $0<k_1<k_2<\dots<k_r\leq N$ with $\Sep_N(k_1, \dots, k_r)=r$ and all
$(i_1,\dots, i_r)\in \{1,\dots, p\}^r,$
$$\left(\prod_{j=1}^r \sigma_{\rho_N, i_j}\right)\left(1-\eps_N\right)
\leq\mu\left(\bigcap_{j=1}^r \fA^{k_j}_{\fr_N, i_j} \right)\leq \left(\prod_{j=1}^r \sigma_{\rho_N, i_j}\right)\left(1+\eps_N\right).$$

\item[$(M2)_r$] For $0< k_1<k_2<\cdots <k_r\leq N$ with $\Sep_N(k_1, \dots, k_r)=m<r$,
$$\mu\left(\bigcap_{j=1}^r \fA^{k_j}_{\fr_N} \right)
\leq {\frac{C\fv_{\fr_N}^m}{(\ln N)^{100r}}}.$$

\item[$(M3)_r$] For $0< k_1<\dots<k_r<l_1<\dots<l_r$ where $2^i < k_\alpha \leq 2^{i+1},$ $2^j < l_\beta \leq 2^{j+1}$ ($1 \leq \alpha, \beta \leq r$), $j - i \geq b$ ($b \geq 2$),
$$\wSep_{2^{i+1}}(k_1,\ldots,k_r)=r, \quad \wSep_{2^{j+1}}(l_1,\ldots,l_r)=r,
 \quad l_1-k_r\geq \hat s(2^{j+1})$$
and for each $(i_1, i_2, \dots, i_r),\,(j_1, j_2,\dots, j_r)\in \{1,\dots, p\}^r,$
$$
\mu\left(\left(\bigcap _{\alpha=1}^r \fA^{k_\alpha}_{\fr_{2^i},i_\alpha}\right)\bigcap
\left( \bigcap_{\beta=1}^r \fA^{l_\beta}_{\fr_{2^j},j_\beta} \right)\right)\leq \left(\prod_{\alpha=1}^r \sigma_{\rho_{2^i}, i_\alpha}\right)\left(\prod_{\beta=1}^r \sigma_{\rho_{2^j}, j_\beta}\right)\left(1+\eps_N\right).$$
\end{itemize}

Define the events
$$A_{*}^N=\left\{\omega\in\Omega:\,\exists\, 0<k_1<k_2<\dots <k_r\leq N\,\,\text{such that}\,\,\Sep_N(k_1, \dots, k_r)=r\,\,\text{and}\,\,\omega\in\fA^{k_j}_{\fr_N}\right\},$$
$$A_{I_1,\dots, I_r; i_1, \dots i_r}^N=\left\{\omega\in\Omega:\,\exists\, k_j\,\,\text{such that}\,\,k_j/N\in  I_j\,\,\text{and }\,\,\omega\in\fA^{k_j}_{\fr_N, i_j} \right\},$$
and
$$H_{I_1,\dots, I_r; i_1, \dots, i_r}=\bigcap_{M=1}^\infty\bigcup_{N=M}^\infty A_{I_1,\dots, I_r; i_1, \dots i_r}^N.$$

Additionally, assume
$$ (M4)_{*} \quad\quad
\limsup_{N\to\infty}
\frac{\mu\left(A_{*}^N\right)}{N^r\sigma_{\rho_N}^r}\leq C
$$
and
$$ (M4)_{I_1,\dots, I_r; i_1, \dots, i_r} \quad\quad
\lim_{N\to\infty}
\frac{\mu\left(A_{I_1, \dots, I_r; i_1, \dots, i_r}^N\right)}{N^r\Pi_{j=1}^r\left(\sigma_{\rho_N,i_j}|I_j|\right)}=1.
$$

\begin{Lem}
\label{LmMultiBC}
Given a family of events $\{\fA_{\rho_N} \}_{N=1}^\infty$ with decompositions  $ \DS \fA_{\rho_N}=\bigcup_{i=1}^p \fA_{\rho_N,i}$, define
$$\bS_r=\sum_{j=1}^\infty \left(2^j {\fv_{\fr_{2^j}}}\right)^r.$$
Assume  $\rho_{N_1}\leq\rho_{N_2}$ and $\fA_{\fr_{N_1}}\subset\fA_{\fr_{N_2}}$ for $N_1\geq N_2.$

\noindent \textbf{Case 1:} $\bS_r<\infty.$
\begin{itemize}
    \item[(a)] If  $(M1)_r$ and $(M2)_r$ hold, then with probability $1,$ the events $\{\fA_{\rho_N}^k\}_{k=1}^N$ occur at most $(r-1)$ times for all sufficiently large $N.$
    \item[(b)] If $(M4)_*$ holds, then with probability 1,  the event $A^N_*$ does not occur for all sufficiently large $N.$
\end{itemize}

\noindent \textbf{Case 2:} $\bS_r=\infty.$ \\
For any $(i_1,\dots, i_r)\in \{1,\dots, p\}^r$ and any intervals $I_1, I_2 \dots I_r\subset [0,1],$
\begin{itemize}
    \item[(c)] If $(M1)_k$ and $(M2)_k$ hold for $k=r,r+1,\dots, 2r,$ and $(M3)_r$ is satisfied under the condition
    $\DS \hat s(N)/N <\min_{0\leq i<j\leq r} d(I_i, I_j)$ where $I_0=0,$ then
    $\DS \mu\left(H_{I_1,\dots, I_r; i_1, \dots, i_r}\right)=1.$
    \item[(d)] If $(M3)_r$ is satisfied under the condition $\DS \hat s(N)/N < \min_{0\leq i<j\leq r} d(I_i, I_j)$ where $I_0 = 0,$ and $(M4)_{I_1, \dots, I_r; i_1, \dots, i_r}$ holds, then
    $\DS \mu\left(H_{I_1,\dots, I_r; i_1, \dots, i_r}\right) = 1.$
\end{itemize}
\end{Lem}

\begin{proof}[Proof of Lemma~\ref{LmMultiBC}]
Part (a) is Theorem 2.6 in \cite{DBS21}.  We now consider the case where $\bS_r<\infty$ and $(M4)_{*}$ holds.
Suppose  $\{\delta_N\}$ is a sequence of positive numbers such that $\rho_N=\delta_{2N}.$ Define the event
$$A_{*,2}^N=\left\{\omega\in\Omega:\,\exists\, 0<k_1<k_2<\dots <k_r\leq 2N,\,\Sep_{2N}(k_1, \dots k_r)=r\,\,\text{such that}\,\,\omega\in\fA^{k_j}_{\fr_N}\right\}.
$$
By $(M4)_{*},$
$\DS \limsup_{N\to\infty}
\frac{\mu\left(A_{*,2}^N\right)}{(2N)^r\sigma_{\delta_{2N}}^r}\leq C.
$
Since $\sigma_{\rho_N}\!\!=\!\!\mu\left(\Omega_{\rho_N}\right)\!\!=\!\!\mu\left(\Omega_{\delta_{2N}}\right)
\!\!=\!\!\sigma_{\delta_{2N}},$ we have
$$\limsup_{N\to\infty}
\frac{\mu\left(A_{*,2}^N\right)}{N^r\sigma_{\rho_{N}}^r}\leq C.
$$
By taking $N=2^M,$ we find that $\mu\left( A_{*,2}^{2^M}\right)\leq C2^{rM} {\fv_{\fr_{2^M}}^r}$ and thus $$\sum_{M=1}^\infty\mu\left( A_{*,2}^{2^M}\right)\leq C\bS_r<\infty.$$ By the classical Borel--Cantelli lemma, $A_{*,2}^{2^M}$ does not occurs for infinitely many $M$'s with probability 1. Then Part (b) follows from the fact that
$A_{*}^N\subset A_{*,2}^{2^M}$ when  $2^M\leq N\leq 2^{M+1}.$

Part (c) is Theorem 2.12 in  \cite{DBS21}. Finally, part (d) holds because in the proof of part (c) in \cite{DBS21},
$(M1)_k$ and $(M2)_k$ for $k=r,r+1,\dots, 2r$  are only used to verify $(M4)_{I_1,\dots, I_r; i_1, \dots, i_r}$.
\end{proof}

We now address the central problem of this paper by analyzing the following targets:
\bea\label{fAro}\fA_{\rho_N}=\left\{y:\phi(y)>C\rho_N\right\}\eea
with $\rho_N=N^{1/s} (\ln N)^t$ where $t>0$ and $C>0;$
\bea\label{fAroj}\fA_{\rho_N,j}=\left\{y:\frac{\phi(y)}{\rho_N} \in [c_j-\eps, c_j+\eps]\right\}\eea
with $c_j>0$, $0<\eps<c_j$ for non-periodic points, and
$0<\eps<\left(\gamma^{q/s}-1\right)c_j/\left(\gamma^{q/s}+1\right)$ for periodic points of period $q$. Also let
\bea\label{tlfAroj}\tilde\fA_{\rho_N,j}=\fA_{\rho_N,j}^c\mathrel{\scalebox{1}{$\bigcap$}} f^{-1}\fA_{\rho_N,j}^c\mathrel{\scalebox{1}{$\bigcap$}}\cdots \mathrel{\scalebox{1}{$\bigcap$}} f^{-(p_0-1)}\fA_{\rho_N,j}^c\mathrel{\scalebox{1}{$\bigcap$}}f^{-p_0}\fA_{\rho_N,j}\eea
where $p_0=\ln\ln N.$

Additionally, we set $s(N) = K \ln N$, where $K$ is chosen to be sufficiently large, and set $\hat s(N) = 2 \varepsilon N$.

\begin{Prop}\label{AdTag}
\begin{itemize}
\item[(a)] The events $\left\{\fA_{\rho_N}\right\}$ and $\left\{\fA_{\rho_N,j}\right\}$ satisfy conditions $(M1)_r$ and $(M3)_r$ for all $r\in\mathbb{N}.$
\item[(b)]  If $x$ is Diophantine then condition $(M2)_r$ holds for any $r\in\mathbb{N}.$
\item[(c)]  Condition\ $(M4)_*$ is satisfied for $\left\{\fA_{\rho_N}\right\}.$
\item[(d)] Suppose that $\min_{i\neq j}d(I_i,I_j)>\eps$ for some $\eps>0.$ If $x$ is not a periodic point, condition $(M4)_{I_1, \dots, I_r; i_1, \dots, i_r}$ holds for the events $\left\{\fA_{\rho_N,j}\right\}.$ Moreover, if $x$ is periodic, the same condition holds when  $\left\{\fA_{\rho_N,j}\right\}$ is replaced by the events  $\left\{\tilde\fA_{\rho_N,j}\right\}$ in the definition of $A_{I_1,\dots, I_r; i_1, \dots i_r}^N.$
\end{itemize}
\end{Prop}

The proof of Proposition~\ref{AdTag} follows a similar structure to that of Proposition 3.9 in \cite{DBS21} with some modifications, which is provided in Appendix~\ref{GTEM}.

Next, we will recall the exponential mixing property for expanding maps, which will be useful not only in the proof Proposition~\ref{AdTag} but also in the next section.
We denote $BV$ to be the space of functions with bounded variation on $\mathbb{S},$ and
$\DS \EXP\psi=\int\psi \rd\mu\,\,\,\text{for}\,\,\,\psi\in L^1(\mu).$
By Proposition 3.8 of \cite{Viana97},
there exist constants $C>0,$ $\theta<1$ such that for all $\psi_1,\,\psi_2\in BV,$
\begin{equation}
\label{PWM2}
\left| \EXP\left(\psi_1(\psi_2 \circ f^n)\right) - (\EXP\psi_1)( E\psi_2)\right| \leq  C
\|\psi_1\|_{BV} \|\psi_2\|_{L^1} \theta^n, \,n\in\mathbb{N},
\end{equation}
 where
$\|\psi\|_{BV}:=\|\psi\|_{L^1}+\text{Var}(\psi),$ and $\text{Var}(\psi)$ denotes the variation of~$\psi.$

\begin{Lem}
\label{LmPEMixBV}
Suppose $\varphi_j\in BV$ for $1\leq j\leq q.$
\begin{enumerate}
\item[(i)] If $\varphi_j\geq 0$  for $1\leq j\leq q,$  then for any $k_1\leq k_2\leq\cdots\leq k_q$ with $k_i\in\mathbb{N},$ we have
\begin{equation}
\label{PWMr}\left(\prod_{j=1}^{q-1} \left(\|\varphi_j\|_{L^1}-C\|\varphi_j\|_{BV}\theta^{k_{j+1}-k_j} \right)\right)\|\varphi_q\|_{L^1}\leq \left\|\prod_{j=1}^q\varphi_j\circ f^{k_j}\right\|_{L^1}
\end{equation}
$$\leq\left(\prod_{j=1}^{q-1} \left(\|\varphi_j\|_{L^1}+C\|\varphi_j\|_{BV}\theta^{k_{j+1}-k_j} \right)\right)\|\varphi_q\|_{L^1}$$
\item[(ii)]  If for each $j,$ $k_{j+1}-k_j\geq k,$ then
\begin{equation}
\label{GPWMr}\left|\EXP\left(\prod_{j=1}^q \varphi_j  \circ f^{k_j }\right)-\prod_{j=1}^q \EXP\varphi_j\right|\leq C\theta^k \prod_{j=1}^q \|\varphi_j\|_{BV}.
\end{equation}
\end{enumerate}
\end{Lem}

Lemma~\ref{LmPEMixBV} could be proved by the arguments of Proposition~6.2 in \cite{DBS21} with $\|\cdot\|_{Lip}$
replaced by $\|\cdot\|_{BV}.$ However, in the present case a simpler argument is available
and we provide it in Appendix~\ref{AppExpMix} for the convenience of the reader.

\subsection{Proof of Theorem~\ref{ThGMLIL}}
\label{ScBCMults}

Fix a small $\delta>0$ and let
$$S_n'=\sum_{1\leq k\leq n} \left(X_k \one_{\{X_k\geq N^{1/s} (\ln N)^\delta\}}-b_N'\right), \quad
S_n''=\sum_{1\leq k\leq n} \left(X_k \one_{\{ X_k< N^{1/s} (\ln N)^\delta\}}-b_N''\right), $$
where $b_N'=b_N''=0$ if $s\leq1$ and $b_N', b_N''$ are chosen to so that the expectations of the left-hand sides are zero when $s>1.$

\begin{Lem}\label{lem5} If $\delta<\alpha$ and either $0<s<1$ or $1\leq s<2$ and $x$ is Diophanitne then
$$ \max_{1\leq n\leq N} \frac{\left|S_n''\right|}{N^{1/s} (\ln N)^\alpha}\rightarrow 0 \quad\text{a.s.}. $$
\end{Lem}

The proof of Lemma~\ref{lem5} is the main part of this paper and is contained in the next section. Given Lemma~\ref{lem5}, we proceed to show Theorem~\ref{ThGMLIL}.

Let
$$
W_N'(t)=\frac{S_{Nt}'}
{N^{1/s} (\ln N)^\alpha}.
$$
By Lemma \ref{lem5}, the limit points of
$W_N(\cdot)$ and $W_N'(\cdot)$ are the same.
Consider the targets $$\fA_{\rho_N}=\{y:\phi(y)>\rho_N\},$$
where $\rho_N=N^{1/s} (\ln N)^{\delta}.$ By Proposition~\ref{AdTag}(c) and Lemma~\ref{LmMultiBC}(b), for sufficiently for large $N,$ there exist at most $\left[(\delta s)^{-1}\right]$ intervals $I_j$ of length $2 s(N)$ such that every
$n$ satisfying
$X_n>N^{1/s} (\ln N)^\delta$ is contained in one of those intervals.

Moreover, by choosing $\bralpha$ such that $((r+1) s)^{-1}<\bralpha<\alpha,$ we see that among those intervals, at most $r$
contain  $n$ for which $$X_n\geq N^{1/s} (\ln N)^{\bralpha}.$$

We claim that the contribution of the remaining intervals is negligible for almost every $y.$
Indeed, suppose there is an interval $I$ of length $2 s(N)$
and $\brn\in I$ such that
\begin{equation}
\label{TwoHumps}
\max_{n\in I} X_n=X_\brn\leq N^{1/s} (\ln N)^\bralpha\;\; \mathrm{and}\;\; \sum_{n\in I} X_n\geq \eps N^{1/s} (\ln N)^\alpha.
\end{equation}
Let
$$I'=\{n\in I: |n-\brn|\leq (\ln N)^{(\alpha-\bralpha)/2}\}, \quad I''=I\setminus I', \quad
 Y_I=\max_{n\in I''}  X_n.$$ Then
$$ \sum_{n\in I} X_n\leq \sum_{n\in I'} X_n+\sum_{n\in I''} X_n\leq N^{1/s}  (\ln N)^{(\alpha+\bralpha)/2}
+2 Y_I s(N).
$$
By the definition of $I,$ we have
$$\DS Y_I\geq \breps N^{1/s}  (\ln N)^{\alpha-1}$$
for some
$\breps=\breps(\eps)>0.$ Accordingly, if the interval $I$ satisfies~\eqref{TwoHumps}, then
there exist two moments $\brn, \hn\in [1, N]$ such that
\begin{equation}
\label{TwoHumps2}
(\ln N)^{(\alpha-\bralpha)/2} \leq  |\brn-\hn|\leq 2 s(N),\quad  \min\left\{X_\brn,\,X_\hn\right\}>\breps N^{1/s}(\ln N)^{\alpha-1}.
\end{equation}
Thus, the following lemma establishes that the contribution of the intervals satisfying \eqref{TwoHumps} is negligible. Its proof is provided in Appendix \ref{GTEM}.
\begin{Lem}
\label{LmTwoHumps}
For almost every $y,$ \eqref{TwoHumps2} happens only for finitely many $N.$
\end{Lem}

It then follows that the limit points of
$W_N'(\cdot)$ have at most~$r$ positive jumps. That is,
 all limit points are in $\bbH_r.$ \smallskip

To verify that all points in $\bbH_r$ are realizable, we fix an
$r$-tuple $(c_1, \dots, c_r)$ such that $c_j>0$ for all $j,$ together with an increasing $r$-tuple $(t_1, \dots, t_r)$ satisfying $t_1 < t_2 < \cdots < t_r.$

If $x$ is not a periodic point,   we define the events
$$\fA_{\rho_N,j}=\left\{y:\frac{\phi(y)}{\rho_N} \in [c_j-\eps, c_j+\eps]\right\},\,\,1\leq j\leq r$$
for $0<\eps<c_j.$

If $x$ is a periodic point of period $q,$ we instead choose
$0<\eps<\left(\gamma^{q/s}-1\right)c_j/\left(\gamma^{q/s}+1\right)$ and consider the events
$$\tilde\fA_{\rho_N,j}=\fA_{\rho_N,j}^c\mathrel{\scalebox{1}{$\bigcap$}} f^{-1}\fA_{\rho_N,j}^c\mathrel{\scalebox{1}{$\bigcap$}}\cdots \mathrel{\scalebox{1}{$\bigcap$}} f^{-(p_0-1)}\fA_{\rho_N,j}^c\mathrel{\scalebox{1}{$\bigcap$}}f^{-p_0}\fA_{\rho_N,j},\,\,1\leq j\leq r,$$
where $p_0=\ln\ln N.$

Taking $\rho_N=N^{1/s} (\ln N)^\alpha$ and applying Proposition~\ref{AdTag}\,(d) and Lemma~\ref{LmMultiBC}\,(d), we conclude that with probability $1$ there is a subsequence $N_k$ and times $n_1(k), n_2(k), \dots, n_r(k)$ such that
$$\frac{n_j(k)}{N_k} \in [t_j-2\eps, t_j+2\eps]\quad\text{and}\quad
\frac{X_{n_j(k)}(y)}{N_k^{1/s} (\ln N_k)^\alpha}\in [c_j-\eps, c_j+\eps] . $$
 It follows that for each
$W\in \bbH_r$,
 $W$ belongs to the limit set with probability $1$ by taking a countable set of $\eps_l$ converging to $0.$ Recall the fact that two elements  $W_1$ and $W_2$ of $\bbH_r$
with exactly $r$ jumps
are $J_1$ close if and only if $\{c_m^1\}_{m=1}^r$ is close to $\{c_m^2\}_{m=1}^r$ and  $\{t_m^1\}_{m=1}^r$ is close to $\{t_m^2\}_{m=1}^r$.
 Taking a countable dense subset in $\bbH_r$
 completes the proof.

\section{Proof of Lemma~\ref{lem5}.}
\label{ScHeavy}

\subsection{Splitting the sum}

The proof of Lemma \ref{lem5} will be given in the next three subsections. Fix a large number $D$ satisfying
\begin{align}
\label{D-SmallS}
D(1-s)+\alpha>1 & \quad\text{if}\quad  s<1; \\
\label{D-LargeS}
D(2-s)+2\alpha> 3 & \quad\text{if}\quad  1\leq s<2.
\end{align}
and split $S_N'' =\bar{S}_N + \bar{\bar S}_N $ where
$\DS \bar{S}_n = \sum_{k=1}^{n}  \bar{X}_k$ and
$\DS \bar{\bar S}_n =   \sum_{k=1}^{n} \bar{\bar{X}}_k$
for
$$\bar{X}_k=X_k \one_{\{X_k<N^{1/s} (\ln N)^{-D} \}}-\bar{b}_N\,\,\,\text{and}\,\,\,\bar{\bar{X}}_k= X_k\one_{ \{ N^{1/s} (\ln N)^{-D} \leq X_k < N^{1/s}(\ln N)^\delta \} }-\bar{\bar{b}}_N,$$
where, as before, $\bar{b}_N=\bar{\bar{b}}_N=0$ for $s\leq1$ and they are chosen so that the expectations of the left-hand sides are zero for $s>1.$

We will show that
\begin{equation}
\label{SmallVal}
\max_{1\leq n\leq N}  \frac{|\brS_n|}{N^{1/s} (\ln N)^\alpha}\to 0\;\; \text{a.s.}
\end{equation}
and
\begin{equation}
\label{IntVal}
\max_{1\leq n\leq N}  \frac{|\brbrS_n|}{N^{1/s} (\ln N)^\alpha}\to 0\;\; \text{a.s.}.
\end{equation}

In \S~\ref{SSS<1}, we consider the case $s<1,$ where the argument is much simpler, since the sums
$\brS_n$ and $\brbrS_n$ are almost monotone. The case $1\leq s<2$ is more involved and requires the
application of the reflection principle (see equations \eqref{EqRefl} and \eqref{EqRefl2}). In \S~\ref{SSLow} and
\S~\ref{SSInt}, we establish
\eqref{SmallVal} and \eqref{IntVal}, respectively, for $1\leq s<2.$

\subsection{The case of infinite expectation.}
\label{SSS<1}

Here we prove \eqref{SmallVal} and \eqref{IntVal} in case $s<1.$

We begin by considering \eqref{SmallVal}.
By \eqref{Tail}, we have
$$\DS E \left|\bar{S}_N\right|\leq N \int_0^{N^{1/s}(\ln N)^{-D} } \frac{a_s\rd t}{t^s}+LN\leq  C N^{1/s} (\ln N)^{-D(1-s)},$$
for $L=\max_{x\in\mathbb{S}}\left|\psi(x)\right|.$
Hence
$$ E\left| \frac{\bar S_N}{ N^{1/s} (\ln N)^{\alpha} } \right| \leq
\frac{C}{(\ln N)^{D(1-s) + \alpha} }.$$
Let $ N_k= 2^k.$ For each $\varepsilon >0$, we have
\begin{equation}
\label{DBound1}
\sum_{k\geq 1} \mu\left( \left| \frac{\bar{S}_{N_k}}{N_k^{1/s} (\ln N_k)^\alpha }  \right| > \varepsilon \right)  \leq \sum_{k\geq 1} \frac{C}{\varepsilon k ^{D(1-s) + \alpha} } < \infty
\end{equation}
where the convergence follows from \eqref{D-SmallS}. By the classical Borel-Cantelli lemma
\begin{equation}
\label{LowTermSubSeq}
 \frac{|\bar{S}_{N_k}|}{N_k^{1/s} (\ln N_k)^\alpha } \rightarrow 0 \quad \text{a.s.}.
\end{equation}
Next, consider the case $2^k < N \leq 2^{k+1}.$ Since
$\DS |\bar{S}_N|\leq |\bar{S}_{N_{k+1}}|+2L N_{k+1}$ and $s<1$, \eqref{LowTermSubSeq}
yields
\begin{equation}
\label{BarSSmall}
  \frac{|\bar{S}_{N}|}{N^{1/s} (\ln N)^\alpha } \rightarrow 0 \quad \text{ a.s.}
\end{equation}
which completes the proof of~\eqref{SmallVal}.

\[
\mu(\hB_{N,j}) \sim 2 a^s h(x) N^{-1} (\ln N)^{-j\delta s}.
\]

Next, we prove \eqref{IntVal}.  Let $D$ be chosen such that $D = q \delta,$ where $q$ is a positive integer.
For $j=0,-1,\cdots,-q,$ define
\bea\label{hBNj} \hB_{N,j}=\left\{y: \phi(y)\in \left(N^{1/s}(\ln N)^{j\delta},\; N^{1/s}(\ln N)^{(j+1)\delta} \right]\right\},\eea
and
\bea\label{bbDnj} \brbrS_{n, j}=\sum_{k=1}^n \phi(f^ky)\one_{\hB_{N,j}}(f^ky),\eea
where $\one.$ is the indicator function.
It suffices to show that for each $j$,
\begin{equation}
  \label{IntValJ}
  \max_{1\leq n\leq N}\frac{ |\brbrS_{n,j}|}{N^{1/s} (\ln N)^\alpha}\to 0.
\end{equation}
Since $|\brbrS_{n,j}|\leq|\brbrS_{N,j}|+2LN$ and $s<1,$
it suffices to prove that
\begin{equation}
  \label{IntValJEnd}
  \frac{|\brbrS_{N,j}|}{N^{1/s} (\ln N)^\alpha}\to 0.
\end{equation}

The next result is proven in Appendix \ref{GTEM}.
\begin{Lem}\label{LmCollar}
\begin{itemize}
\item[(a)] If $x$ is not a periodic point then
$ \DS \EXP\brbrS_{N,j}^m\leq C N^{m/s}(\ln N)^{\delta m}.$
\item[(b)] If $x$ is periodic then
$\DS \EXP\brbrS_{N,j}^m\leq C N^{m/s}(\ln N)^{\delta m}\left(\ln\ln N\right)^m.$
\end{itemize}
\end{Lem}

By Markov's inequality
$$ \mu\left(|\brbrS_{N,j}|\geq \eps N^{1/s} (\ln N)^\alpha\right)\leq \frac{C\left(\ln\ln N\right)^m}{\eps^m (\ln N)^{(\alpha-\delta)m}}$$
for even $m.$ Taking $m$ such that $m(\alpha-\delta)>1$ and
arguing as in the proof of \eqref{BarSSmall} we obtain
\begin{equation}
\label{BarBarSSmall}
\frac{|\bar{\bar S}_{N, j}|}{N^{1/s} (\ln N)^\alpha}\rightarrow 0 \quad \text{a.s.}.
\end{equation}
Since $j$ is arbitrary, we obtain \eqref{IntVal} for $s<1.$

\subsection{Controlling small values of $X$ for $1\leq s<2$}
\label{SSLow}

We now prove  \eqref{SmallVal} for $1\leq s<2.$

Given $y\in \mathbb{S},$ let $C_n(y)$ be the continuity component of
$\phi 1_{\phi\leq N^{1/s}(\ln N)^{-D}}\circ f^k$ for all $k\leq n,$ containing $y.$ Then $f^n|_{C_n(y)}$ be a diffeomorphism. Define $I_n(y)=f^n C_n(y).$
We say that $C_n(y)$ is {\it short} if the length of $I_n(y)$ is less than $N^{-1}(\ln N)^{-100},$
and {\it long} otherwise.

Fix $\varepsilon>0.$ We say that $C_n$ is {\it extremal} if there exists $z\in C_n$ such that
$$\left|\brS_n(z)\right|\geq \eps N^{1/s} (\ln N)^\alpha$$ and for each $k<n,$
$\DS \left|\brS_k(\cdot)\right|<\eps N^{1/s} (\ln N)^\alpha$ on $C_n.$
Note that different extremal components are disjoint.  Moreover, we have
\begin{equation}
\label{LargeEverywhere}
 \left|\brS_n(y)\right|\geq \frac{3}{4}  \eps N^{1/s} (\ln N)^\alpha \quad \text{for all }y \text{ in }C_n.
\end{equation}

Indeed, fix a large $K$ and write
$$ \left|\brS_n(y)-\brS_n(z)\right|\leq \sum_{k=1}^{n-K\ln N} \left|\brX_k(y)-\brX_k(z)\right|+
\left|\brS_{K\ln N}(f^{n-K\ln N} y)\right|+\left|\brS_{K\ln N}(f^{n-K\ln N} z)\right|.$$
The last two terms are smaller than $2K N^{1/s} (\ln N)^{1-D}$ due to the cutoff in the definition of
$\brS_n.$ To estimate the first term, observe that since $d(f^n y, f^n z)\leq 1$ and $f$ is expanding,
for $k\leq n-K\ln N,$  we have $d(f^k y, f^k z)\leq \theta^{K\ln N}.$ Since
$\phi \one_{\{ \phi<N^{1/s} (\ln N)^{-D} \}}$ is Lipschitz on its continuity components
with a Lipschitz constant less than $C N^{1+1/s}(\ln N)^{-(1+s)D}$,
it follows that if $K$ is sufficiently large, then
$$\left|\brX_k(y)-\brX_k(z)\right|\leq N^{-100},$$
which proves  \eqref{LargeEverywhere}.

Let $\cE_N$ be the set of points which
belong to some extremal component for $n\leq N.$

\begin{Lem}
\label{LmExtreme}
If $D$ is  is sufficiently large, then $\mu(\cE_N)\leq {C}{(\ln N)^{-2}}.$
\end{Lem}

Lemma \ref{LmExtreme} implies \eqref{SmallVal}. In fact, combining this lemma with the classical Borel-Cantelli lemma
we conclude that with probability $1$, for sufficiently large $k$, we have
$$\left|\brS_n\right|<\eps N_k^{1/s}(\ln N_k)^\alpha\,\,\,\text{for}\,\,\, n \leq N_k, \,\,\,N_k = 2^k.$$
Since $\eps$ is arbitrary, \eqref{SmallVal} follows.

\begin{proof}[Proof of Lemma \ref{LmExtreme}] Note that $\phi \one_{\{\phi< N^{1/s} (\ln N)^{-D}\}}$
is only discontinuous at two points $z_{-, N}$ and $z_{+, N}$ where $\phi(z_{\pm,N})=N^{1/s} (\ln N)^{-D}.$
Therefore, if $y$ belongs to a short component, there exists $k\leq n$ such that
$$ d\left(f^k y, \{z_{-,N}, z_{+, N}\}\right)\leq N^{-1}(\ln N)^{-100} . $$
Since for a fixed $k,$
$\DS \mu\left(\left\{y:\,d\left(f^k y, \{z_{-,N}, z_{+, N}\}\right)\leq N^{-1}(\ln N)^{-100}\right\}\right)=O\left(N^{-1}(\ln N)^{-100}\right), $
we obtain that the measure of points which belong to short components is
$\DS O\left((\ln N)^{-100}\right).$ It remains to estimate the measure of points which belong to long extreme components.
Let $C_n$ be one such component and $I_n=f^n C_n.$ We will prove two estimates:
\begin{equation}
  \label{EqRefl}
\mu\left(y\in C_n: |\brS_N(y)|\geq \frac{\eps}{2} N^{1/s} (\ln N)^\alpha\right) \geq C \mu(C_n)
\end{equation}
for some $C>0$ independent of $N,$ $C_n$ and
\begin{equation}
  \label{NkCheb}
\mu\left(y\in \mathbb{S}: |\brS_N(y)|\geq \frac{\eps}{2} N^{1/s} (\ln N)^\alpha\right)\leq \frac{C}{(\ln N)^2}
\end{equation}
for some $C>0$ independent of $N.$

To prove \eqref{NkCheb} we use Chebyshev's inequality. By \eqref{Tail} we have
\begin{equation}
\label{S<1TrV}
 \Var(\brX_n)\leq C N^{2/s-1} (\ln N)^{{-D(2-s)}}.
\end{equation}
We now estimate
\begin{equation} \label{CovEst}
\left|\text{Cov}(\brX_n, \brX_m)\right|\leq \begin{cases}
O\left(N^{2/s-1} (\ln N)^{{-D(2-s)}}\right) &
\text{if }|m-n|<K\ln N, \\
N^{-2} & \text{if }|m-n|\geq K\ln N, \end{cases}
\end{equation}
where the first bound follows from \eqref{S<1TrV} and Cauchy-Schwartz inequality and the second bound follows from \eqref{PWM2}.

Summing over all $m, n$ we get
\begin{equation}
\label{DBound2}
 \Var(\brS_N)\leq C\left(N^{2/s} (\ln N)^{{ 1-D(2-s)}}+1\right).
\end{equation}
Since $D(2-s)+ 2\alpha>3$ due to \eqref{D-LargeS},
\eqref{NkCheb} follows by Chebyshev's inequality.

The proof of \eqref{EqRefl} follows similar ideas. Namely, in view of
\eqref{LargeEverywhere},
it suffices to show that
there exists $\tilde{C}<1$ such that
\begin{equation}
  \label{ReflLoc}
  \mu\left(y\in C_n: \left|\sum_{k=n+1}^N \brX_k(y)\right|>\frac{\eps}{ 4} N^{1/s} (\ln N)^\alpha\right)
  \leq \tilde{C} \mu(C_n).
\end{equation}

Since $f^n$ has bounded distortion property, \eqref{ReflLoc} follows from
\begin{equation}
  \label{ReflTCh}
\mu\left(z\in I_n: |\brS_{N-n}(z)|>   \frac{\eps}{ 4} N^{1/s} (\ln N)^\alpha\right)
  \leq \tilde{C} \mu(I_n).
\end{equation}
Split $\brS_{N-n}=\brS_{K\ln N}+(\brS_{N-n}-\brS_{K\ln N}).$
The first term is bounded by
\begin{equation}
\label{Buffer}
  O\left(N^{1/s} (\ln N)^{1-D}\right).
\end{equation}
Note that
$$
\EXP\left((\brS_{N-n}-\brS_{K\ln N})\one_{I_n}\right)^2=
\sum_{K\ln N<k_1\leq k_2\leq N-n}\EXP\left(\bar{X}_{k_1}\bar{X}_{k_2}\one_{I_n}\right). $$
Recall that $\brX_{k_1}=\phi_N\circ f^{k_1}$ for $\phi_N=\phi\one_{\{\phi<N^{1/s}(\ln N)^{-D}\}}.$
Thus denoting $p=K\ln N/2,$ we obtain
$$ \EXP\left(\bar{X}_{k_1}\bar{X}_{k_2}\one_{I_n}\right)=\EXP\left((\phi_N\circ f^{k_1-p})  (\phi_N\circ f^{k_2-p})  \right) \mu(I_n) +O\left(\left\|(\phi_N\circ f^{k_1-p})  (\phi_N\circ f^{k_2-p})\right\|_{L_1}\eta_{n,p}\right)
$$
by \eqref{PWM2}, where
$$ |\eta_{n,p}|\leq C\|I_n\|_{BV}\theta^p\leq C \theta^p. $$
If $K$ and hence $p$ is large enough then $|\eta_{n,p}| \leq N^{-(2/s+100)}$
and so
$$ \sum_{K\ln N\leq k_1\leq k_2\leq N-n}
\left\|(\phi_N\circ f^{k_1-p})  (\phi_N\circ f^{k_2-p})\right\|_{L_1}\eta_{n,p}=O\left(N^{-98}\right). $$
On the other hand, similarly to \eqref{DBound2} we get
$$ \sum_{K\ln N\leq k_1\leq k_2\leq N-n}
\EXP\left((\phi_N\circ f^{k_1-p})  (\phi_N\circ f^{k_2-p})  \right) \mu(I_n)\leq C\left(N^{2/s} (\ln N)^{{1-D(2-s)}}+1\right) \mu(I_n).  $$

It follows that
$$ \int_{I_n} \left(\sum_{k=K\ln N}^{N-n} \brX_k (z)\right)^2 \rd\mu(z)\leq C \mu(I_n) \left[N^{2/s}
(\ln N)^{1-D(2-s)}+1\right]. $$
Hence
\begin{equation}
\label{ChebLoc}
\mu\left(z\in I_n:\left| \sum_{k=K\ln N}^{N-n} \brX_k(z)\right|> \frac{\eps}{4} N^{1/s} (\ln N)^\alpha\right)
 \leq \frac{16 C \mu(I_n)}{\eps^2} (\ln N)^{1-D(2-s)-2\alpha}.
\end{equation}

Combining \eqref{Buffer} and \eqref{ChebLoc} we obtain \eqref{EqRefl}.
\end{proof}

\subsection{Controlling intermediate values of $X$ for $1\leq s<2$}
\label{SSInt}

Here we prove \eqref{IntVal} in case $1\leq s<2.$

The proof of \eqref{IntVal} for $s\geq1$ will proceed similarly to \S \ref{SSLow}.
Namely we define extremal components similarly to \S \ref{SSLow} by replacing  $\phi \one_{\{\phi< N^{1/s} (\ln N)^{-D}\}}$ with  $\phi \one_{\{N^{1/s} (\ln N)^{-D}\leq\phi< N^{1/s} (\ln N)^\delta\}}$ and
$\brS_n$ with $\brbrS_n.$
We need to obtain the following analogues of
\eqref{EqRefl} and \eqref{NkCheb}
for long extremal components $C_n$:
\begin{equation}
  \label{EqRefl2}
\mu\left(y\in C_n: |\brbrS_N(y)|>\frac{\eps}{2} N^{1/s} (\ln N)^\alpha\right) \geq C \mu(C_n)
\end{equation}
for some $C$ independent of $N,$ $C_n$ and
\begin{equation}
  \label{NkCheb2}
\mu\left(y\in \mathbb{S}: |\brbrS_N(y)|\geq \frac{\eps}{2} N^{1/s} (\ln N)^\alpha\right)\leq \frac{C}{(\ln N)^2}
\end{equation}
for some $C$ independent of $N.$

 In addition, we claim that \eqref{LargeEverywhere} still holds for almost every $y$ in $C_n,$ even though the proof
needs to be modified slightly. Namely it is not true that each summand in
$\brbrS_{K\ln N}(y)$ is smaller than $N^{1/s} (\ln N)^{-D}.$ However, we claim that for almost every $y,$
\bea\label{brbrXLarG}
\sharp\left\{k:\,1\leq k\leq K\ln N,\,\,\brbrX_k(y)\geq N^{1/s}(\ln N)^{-D} \right\}\leq (\ln N)^{(\alpha-\delta)/2}
\eea
for $N$ large enough. If  \eqref{brbrXLarG} is not satisfied, then
\bea\label{TwoHumpsg1}\quad\quad\exists\,\brn,\,\hn\leq K\ln N,\,\,|\brn-\hn|\geq(\ln N)^{(\alpha-\delta)/2},\,\min\left\{\brbrX_\brn,\, \brbrX_\hn\right\}>\breps N^{1/s}(\ln N)^{-D}.\eea The same argument of Lemma~\ref{LmTwoHumps} shows that for almost every $y,$ \eqref{TwoHumpsg1} occurs for finitely many $N.$ Since  the value of each term which is larger than $N^{1/s} (\ln N)^{-D}$ is smaller than $N^{1/s} (\ln N)^\delta,$ we conclude that for almost every $y,$
\begin{equation}
\label{FewLargeTerms}
\left|\brbrS_{K\ln N}(y)\right|\leq C N^{1/s} (\ln N)^{(\delta+\alpha)/2},
\end{equation}
which is sufficient to obtain \eqref{LargeEverywhere}.

Equation \eqref{NkCheb2} follows from the Markov's inequality and the following estimate.

\begin{Lem}
Suppose $x$ is Diophantine. For each positive integer $m,$ we have
\begin{equation}
\label{HT-IntMom}
 E\left(\brbrS_N^m\right)\leq C_m N^{m/s} (\ln N)^{\delta m}.
 \end{equation}
\end{Lem}

\begin{proof}
Recall that we take
$$ \bar{\bar{X}}_n= X_n\one_{ \{ N^{1/s} (\ln N)^{-D} \leq X_n < N^{1/s}(\ln N)^\delta \} }-\bar{\bar{b}}_N,\quad
\bar{\bar{b}}_N=E\left(\phi\one_{ \{ N^{1/s} (\ln N)^{-D} \leq \phi < N^{1/s}(\ln N)^\delta \}}\right)$$
for $1<s<2$ and
$$ \bar{\bar{X}}_n= X_n\one_{ \{ N^{1/s} (\ln N)^{-D} \leq X_n < N^{1/s}(\ln N)^\delta \} }$$
for $s=1.$

Thus
\begin{equation}
\label{SumPowers}
E\left(\brbrS_N^m\right)=\sum_{k=1}^m\sum_{1\leq n_1<\dots<n_k\leq N}\sum_{m_1+\dots+m_k=m} c_{m_1,\dots, m_k}
 E(\bar{\bar{X}}_{n_1}^{m_1}\dots \bar{\bar{X}}_{n_k}^{m_k}),
\end{equation}
where $c_{m_1,\dots, m_r}$ are combinatorial factors.

Note that by \eqref{Tail},
\begin{equation}
\label{MomentAsym}
\EXP \left|\bar{\bar{X}}_0\right|^l
\leq\left\{
\begin{array}{ll}
C N^{l/s-1}(\ln N)^{(s-l)D}, & l<s,\\
C N^{l/s-1}(\ln N)^{(l-s)\delta}, & l>s,\\
C\ln\ln N, & l=s,\\
\end{array}
\right.
\end{equation}
and
$$\left\|\bar{\bar{X}}^l_0\right\|_{BV} \leq C N^{l/s } (\ln N)^{l\delta}.$$

(i) Suppose that $\Sep_N(n_1, \dots, n_k)=k.$ In this case, we apply the multiple mixing bound from
Lemma~\ref{LmPEMixBV}(ii).
We choose $K$ in the definition of separation
condition large enough so that the error term in Lemma~\ref{LmPEMixBV}(ii) is $O(N^{-k}).$
This leads to two further subcases.

 (a) $\DS \min_j m_j>1.$
 Using the second estimate in \eqref{MomentAsym}, we find that
 $$ \prod_j \EXP\left(\bar{\bar{X}}_{n_j}^{m_j}\right)\leq C N^{m/s-k} (\ln N)^{m \delta}, $$
 while the error term is much smaller.
 Summing over all $n_1, \dots, n_k,$ we see that this case gives the leading contribution of
 $C N^{m/s} (\ln N)^{m \delta}.$

(b) $\DS \min_j m_j=1.$ If $1<s<2,$ then $\DS \prod_j \EXP\left(\bar{\bar{X}}_{n_j}^{m_j}\right)=0,$
so each term is $O(N^{-k}),$ and the total contribution of all terms is $O(1).$

If $s=1,$ we have $$\prod_j \EXP\left(\bar{\bar{X}}_{n_j}^{m_j}\right)=\left(\prod_{m_j\geq2} \EXP\left(\bar{\bar{X}}_{n_j}^{m_j}\right)\right)\left(\prod_{m_j=1} \EXP\bar{\bar{X}}_{n_j}\right)$$
This leads to
\Bea\prod_j \EXP\left(\bar{\bar{X}}_{n_j}^{m_j}\right)\leq C\left(\prod_{m_j\geq2}N^{m_j-1}\left(\ln N\right)^{(m_j-1)\delta}\right)\left(\ln\ln N\right)^{m-k}
 \leq C N^{m-k}\left(\ln N\right)^{(m-k)\delta}\left(\ln\ln N\right)^m
\Eea
for $k<m$ and
$$\prod_j \EXP\left(\bar{\bar{X}}_{n_j}^{m_j}\right)=\prod_j \EXP\bar{\bar{X}}_{n_j}
 \leq C\left(\ln\ln N\right)^m
$$
for $k=m.$ It follows that the total contribution of all terms is bounded by  $C N^{m} (\ln N)^{m \delta}$ when  $\DS \min_j m_j=1$ and $s=1.$

(ii) Suppose that $\Sep_N(n_1, \dots, n_k)=\bar{k}<k.$  Given that $x$ is Diophantine,  we consider the case where
$$\min_{j}\{n_{j+1}-n_j\}\geq\epsilon\ln N$$ for some $\epsilon>0.$ Otherwise, the contribution to the sum is $0.$
By \eqref{MomentAsym},
$$\EXP \left|\bar{\bar{X}}_{n_j}\right|^{m_j}\leq C N^{m_j/s-1}(\ln N)^{m_jD}.$$
Additionally, using Lemma~\ref{LmPEMixBV}(i), we find that $$E(\bar{\bar{X}}_{n_1}^{m_1}\dots \bar{\bar{X}}_{n_k}^{m_k}) =O\left(N^{(m/s)-\bar k-\eta} (\ln N)^{Dm}\right),$$
where $\eta=(k-\bar{k})\epsilon |\ln \theta|.$ Since the number of such terms in \eqref{SumPowers} is
 $O\left(N^{\bar{k}} (\ln N)^{k-\bar k}\right),$ we conclude that the contribution of these non-separated terms
 is $\DS O\left(N^{m/s-\eta}(\ln N)^{(D+1)m}\right),$ which is much smaller than
 the  right-hand side of \eqref{HT-IntMom}.
\end{proof}

The proof of \eqref{EqRefl2} is similar to  the proof of \eqref{EqRefl}. Namely,
as in \S\ref{SSLow}, it suffices to establish the following analogue of \eqref{ReflTCh}:
\begin{equation}
  \label{ReflTCh2}
\mu\left(z\in I_n: |\brbrS_{N-n}(z)|>   \frac{\eps}{2} N^{1/s} (\ln N)^\alpha\right)
  \leq C\mu(I_n),
\end{equation}
for some constant $C>0.$
We split $\brbrS_{N-n}=\brbrS_{K\ln N}+\left(\brbrS_{N-n}-\brbrS_{K\ln N}\right).$
The first term is estimated using \eqref{FewLargeTerms}.
For the second term, we proceed as in the proof of \eqref{ReflTCh},  but replace \eqref{DBound2} with
\eqref{HT-IntMom}. This yields
$$
\EXP\left(\left(\brbrS_{N-n}-\brbrS_{K\ln N}\right)^{m}\one_{I_n}\right)
\leq C N^{m/s} (\ln N)^{\delta m}\mu(I_n).
$$
Therefore, we have
\begin{equation}
\label{ChebLocInter2}
\mu\left(z\in I_n: \left|\brbrS_{N-n}(z)-\brbrS_{K\ln N}(z)\right|> \frac{\eps}{4} N^{1/s} (\ln N)^\alpha\right)
\leq \frac{16 C \mu(I_n)}{\eps^2} (\ln N)^{-m(\alpha-\delta)}.
 \end{equation}

Finally, combining \eqref{FewLargeTerms} with \eqref{ChebLocInter2} yields \eqref{ReflTCh2}.

\appendix

\section{Exponential  mixing of all orders for expanding maps}
\label{AppExpMix}
\begin{proof}[Proof of Lemma~\ref{LmPEMixBV}]
We prove part (i) by induction. When $q=2$ the statement reduces to \eqref{PWM2}. For $q>2$
we apply \eqref{PWM2} with $\psi_1=\varphi_1,$  $\psi_2=\prod_{j=2}^{q} (\varphi_j\circ f^{k_j-k_2}),$
and obtain
$$ \left\|\prod_{j=1}^q\left(\varphi_j\circ f^{k_j}\right)\right\|_{L^1}= \|\varphi_1\|_{L^1}\|\psi_2\|_{L^1} +O\left(\theta^{k_2-k_1}\|\varphi_1\|_{BV} \|\psi_2\|_{L^1} \right),$$
By the induction hypothesis, we have
$$\|\psi_2\|_{L^1}\leq
\left(\prod_{j=2}^{q-1}
\left(||\varphi_j||_{L^1}+\theta^{k_j-k_{j-1}} \|\varphi_j\|_{BV}\right)\right)\|\varphi_q\|_{L^1}.
$$
It follows that
$$\left\|\prod_{j=1}^q\left(\varphi_j\circ f^{k_j}\right)\right\|_{L^1}\leq
\left(\prod_{j=1}^{q-1} \left(\|\varphi_j\|_{L^1}+C\|\varphi_j\|_{BV}\theta^{k_{j+1}-k_j} \right)\right)\|\varphi_q\|_{L^1}.$$
Proceeding in a similar way, we can also prove that
$$\left\|\prod_{j=1}^q\left(\varphi_j\circ f^{k_j}\right)\right\|_{L^1}\geq
\left(\prod_{j=1}^{q-1} \left(\|\varphi_j\|_{L^1}-C\|\varphi_j\|_{BV}\theta^{k_{j+1}-k_j} \right)\right)\|\varphi_q\|_{L^1}.$$

Now we turn to the proof of part (ii). When  $q=2,$ the statement reduces to \eqref{PWM2}.
Applying \eqref{PWM2} with
$\psi_1=\varphi_1$ and $\DS \psi_2=\prod_{j=2}^{q} (\varphi_j\circ f^{k_j-k_2}),$ we obtain
\begin{eqnarray*}
&&\left|\EXP\left(\prod_{j=1}^q (\varphi_j  \circ f^{k_j })\right)-\left(\EXP\varphi_1\right)
\EXP\left(\prod_{j=2}^q (\varphi_j  \circ f^{k_j })\right)\right|\\
 &\leq& C \theta^k \|\varphi_1\|_{BV} \left\|\left(\prod_{j=2}^q (\varphi_j  \circ f^{k_j })\right)\right\|_{L_1}
\leq C \theta^k \|\varphi_1\|_{BV} \left\|\left(\prod_{j=2}^q (|\varphi_j|  \circ f^{k_j })\right)\right\|_{L_1}
\leq C\theta^k \prod_{j=1}^q\|\varphi_j\|_{BV},\\
\end{eqnarray*}
where the last inequality follows from  part (i) of the lemma, which we have already proven.
Applying the inductive estimate to
$ \EXP\left(\prod_{j=2}^q (\varphi_j  \circ f^{k_j })\right)$
we obtain \eqref{GPWMr}.
\end{proof}

\section{Good targets for expanding maps}
\label{GTEM}
\begin{proof}[Proof of Proposition~\ref{AdTag}]
We use the notation  $\fA_{\rho_N}$  and $\fA_{\rho_N,j}$ as defined in \eqref{fAro} and \eqref{fAroj}, respectively.

We define the annular set
$\DS B(x,r_1,r_2)=\{y\in\mathbb{S}:r_1\leq|y-x|\leq r_2\},$ and let $\tilde{A}^+_{\rho_N,j},$ $\tilde{A}^-_{\rho_N,j}$  be the characteristic functions of the sets
$$B\left(x,\left(\frac{a}{\rho_N(c_j+\varepsilon)+L}\right)^s,\left(\frac{a}{\rho_N(c_j-\varepsilon)-L}\right)^s\right)
$$
and
$$
B\left(x,\left(\frac{a}{\rho_N(c_j+\varepsilon)-L}\right)^s,\left(\frac{a}{\rho_N(c_j-\varepsilon)+L}\right)^s\right),$$
respectively, where $\DS L=\max_{x\in\mathbb{S}}\left|\psi(x)\right|.$ Then
$$\tilde{A}^-_{\rho_N,j}\leq \one_{\fA_{\rho_N,j}}\leq\tilde{A}^+_{\rho_N,j}, \quad
\EXP\tilde{A}^+_{\rho_N,j}-\EXP\tilde{A}^-_{\rho_N,j}\leq C\rho_N^{-(s+1)}\text{ and }\left\|\tilde{A}^\pm_{\rho_N,j}\right\|_{BV}\leq 2\mu(\fA_{\rho_N,j})+4.$$

We now verify the conditions $(M1)_r,$ $(M2)_r,$ $(M3)_r,$ $(M4)_*$ and $(M4)_{I_1,\dots, I_r; i_1, \dots, i_r}$ using Lemma~\ref{LmPEMixBV}.

(a) We prove $(M1)_r$ and $(M3)_r$ for $\fA_{\rho_N,j},$  and the proof for $\fA_{\rho_N}$ follows in the same approach. Suppose $0<k_1<k_2<\dots<k_r\leq N$ are such that $k_{i+1}-k_i\geq K\ln N,$ where $K$ is a sufficiently large constant, and each $(i_1,\dots, i_r)\in \{1,\dots, p\}^r.$ Then we have
\Bea
&&\mu\left(\bigcap_{j=1}^r \fA^{k_j}_{\fr_N, i_j} \right)\leq
\EXP\left(\prod_{j=1}^r \tilde{A}^+_{\rho_N,i_j}\circ f^{k_j} \right)\\
&\leq&\left(\prod_{j=1}^{r-1} \left(\EXP\tilde{A}^+_{\rho_N,i_j}+C\theta^{K\ln N}\left(2\EXP\tilde{A}^+_{\rho_N,i_j}+4\right)\right)\right)\EXP\tilde{A}^+_{\rho_N,i_r}
\leq\left(\prod_{j=1}^r \sigma_{\rho_N, i_j}\right)(1+\eps_N)
\Eea
where $\eps_N\to 0$ as $N\to\infty.$ This establishes the right-hand side of $(M1)_r.$   The proof for the left-hand side follows similarly, replacing $\tilde{A}^+_{\rho_N,i_j}$ by $\tilde{A}^-_{\rho_N,i_j}$.

Next we prove $(M3)_r.$  Suppose $ k_1<\dots<k_r<l_1<\dots<l_r$ are such that  $2^i<k_\alpha\leq 2^{i+1}$ and $2^j<l_\beta\leq 2^{j+1}$ for $1\leq\alpha,\beta\leq r.$ For each pairs $(i_1, i_2, \dots, i_r)$ and $(j_1, j_2\dots j_r)$ in $\{1,\dots, p\}^r,$ we have
\Bea
&&\mu\left(\left(\bigcap _{\alpha=1}^r \fA^{k_\alpha}_{\fr_{2^i},i_\alpha}\right)\bigcap
\left(\bigcap_{\beta=1}^r \fA^{l_\beta}_{\fr_{2^j},j_\beta} \right)\right)
\leq \EXP\left(\left(\prod_{\alpha=1}^r \tilde{A}^+_{\rho_{2^i},i_\alpha}\circ f^{k_\alpha} \right)\left(\prod_{\beta=1}^r \tilde{A}^+_{\rho_{2^j},j_\beta}\circ f^{k_\beta} \right)\right)\\
&\leq&
\left(\prod_{\alpha=1}^r \left(\EXP\tilde{A}^+_{\rho_{2^i},i_\alpha}+C\theta^{Ki}\left(2\EXP\tilde{A}^+_{\rho_{2^i},i_\alpha}+4\right)\right)\right)
\left(\prod_{\beta=1}^{r-1} \left(\EXP\tilde{A}^+_{\rho_{2^j},j_\beta}
+C\theta^{Kj}\left(2\EXP\tilde{A}^+_{\rho_{2^j},j_\beta}+4\right)\right)\right)\EXP\tilde{A}^+_{\rho_{2^j},j_r}\\
&\leq&\left(\prod_{\alpha=1}^r \sigma_{\rho_{2^i}, i_\alpha}\right)\left(\prod_{\beta=1}^r \sigma_{\rho_{2^j}, j_\beta}\right)(1+\eps_i),
\Eea
which proves $(M3)_r.$

(b) We prove $(M2)_r$ under the assumption that $x$ is Diophantine. Let $\Sep_N(k_1, \dots k_r)=m<r.$ Since $x$ is Diophantine, there exists $\epsilon>0$ such that
$$B\left(x,c\rho_N^{-s}\right)\cap f^{-k}B\left(x,c\rho_N^{-s}\right)=\emptyset$$
for positive integer $k\leq \epsilon\ln N,$ where $c=2a^s/c_j^s.$
Therefore, we assume that $k_{j+1}-k_j \geq \epsilon \ln N$ for all $j.$
Then we obtain
\Bea
 \mu\left(\bigcap_{j=1}^r \fA^{k_j}_{\fr_N} \right)
&\leq& C\left(\prod_{k_{j+1}-k_j>K\ln N}  \sigma_{\rho_N}\right)(\theta^{\epsilon\ln N})^{r-m}\leq{\frac{C \fv_{\fr_N}^m}{(\ln N)^{100r}}}.\\
\Eea

(c) Recall that
\begin{eqnarray*}
A_{*}^N&=&\left\{\omega\in\Omega:\,\exists\, 0<k_1<k_2<\dots <k_r\leq N\,\,\text{such that}\,\,\Sep_N(k_1, \dots, k_r)=r\,\,\text{and}\,\,\omega\in\fA^{k_j}_{\fr_N}\right\}\\
&=&\bigcup_{\overset{0<k_1<\dots <k_r\leq N}{ { \Sep_N}(k_1,\dots, k_r)=r}}\left(\bigcap_{j=1}^r\fA^{k_j}_{\fr_N}\right).
\end{eqnarray*}
We define $$\tau_{*}^N=\sum_{\overset{0<k_1<\dots <k_r\leq N}
{ { \Sep_N}(k_1,\dots, k_r)=r}}\mu\left(\bigcap_{j=1}^r \fA^{k_j}_{\fr_N} \right).$$
Using $(M1)_r$ which has been established, we obtain
$$\lim_{N\rightarrow\infty}\frac{\tau_{*}^N}{\left(N\sigma_{\rho_N}\right)^r}=1.$$
In addition, since $\mu(A_{*}^N)\leq\tau_{*}^N,$  we have
$$\limsup_{N\rightarrow\infty}\frac{\mu(A_{*}^N)}{\left(N\sigma_{\rho_N}\right)^r}\leq 1.$$

(d) We first assume that the point $x$ is not periodic and define
\bea\label{btau}\bar\tau_{*}^N=\sum_{k_j/N\in I_j}\mu\left(\bigcap_{j=1}^r \fA^{k_j}_{\fr_N, i_j} \right).\eea
Recall that
\Bea
A_{I_1,\dots, I_r; i_1, \dots i_r}^N&=&\left\{\omega\in\Omega:\,\exists\, k_j\,\,\text{such that}\,\,k_j/N\in  I_j\,\,\text{and }\,\,\omega\in\fA^{k_j}_{\fr_N, i_j} \right\}\\
&=&\bigcup_{k_j/N\in I_j}\left(\bigcap_{j=1}^N\fA^{k_j}_{\fr_N, i_j}\right).
\Eea

Using $(M1)_r$ which has been proved, we obtain $$\lim_{N\rightarrow\infty}\frac{\bar\tau_{*}^N}{N^r\Pi_{j=1}^r\left(\sigma_{\rho_N,i_j}|I_j|\right)}=1.$$
Therefore, it suffices to show that
$$\lim_{N\rightarrow\infty}\frac{\mu\left(A_{I_1,\dots, I_r; i_1, \dots i_r}^N\right)}{\bar\tau_{*}^N}=1.$$

By Boole's inequality, we have
$$\mu\left(A_{I_1,\dots, I_r; i_1, \dots i_r}^N\right)\leq\bar\tau_{*}^N$$
and
\begin{equation}
\label{M4Lower}
 \mu\left(A_{I_1,\dots, I_r; i_1, \dots i_r}\right)\geq\bar\tau_{*}^N-\sum_{t=r+1}^{2r}\bar\tau_t^N,
\end{equation}
where $$\bar\tau_t^N=\sum_{\overset{0<k_1<\dots <k_r\leq N}
{ { \Sep_N}(k_1,\dots, k_r)=r}}\mu\left(\fA^{k_1}_{\fr_N}\cap \dots \cap \fA^{k_r}_{\fr_N} \cap
\fA^{k_{r+1}}_{\fr_N}\cap\cdots\cap\fA^{k_t}_{\fr_N}\right),$$
because  $\DS \fA_{\rho_N,i_j}\subset\fA_{\rho_N}=\{y:\phi(y)>C\rho_N\}$
for all $i_j,$ $1\leq j\leq r,$ and some $C>0.$

We prove that
$$\frac{\bar\tau_{r+1}^N}{\bar\tau_{*}^N}\to 0\,\,\,\text{as}\,\,\,N\to\infty.$$
The proof of $\bar\tau_{t}^N/\bar\tau_*^N\to 0$ for  other values of $t$ follows similarly.  We consider several cases depending on the size of
$\DS p=\min_{j=1,\dots, r} |k_{r+1}-k_j|$.

(i) Since $x$ is non-periodic, for each fixed $p_0,$ we have
$\DS \fA_{\fr_N} \mathrel{\scalebox{1}{$\bigcap$}} f^{-p} \fA_{\fr_N} =\emptyset$ for all $1\leq p< p_0,$ provided that  $N$ is large enough. Therefore, the contribution from terms with $p< p_0$ is zero.

(ii) Next, suppose that $p_0\leq p\leq \varepsilon\ln N.$ Specifically, we may assume that $k_{r+1}=k_\brj+p$ for some $\brj\leq r$ (the case where $k_{r+1}=k_\brj-p$ is analogous). Define $\bar{A}_{\fr_N}^+$ as the characteristic function of the set
$$B\left(x,\left(\frac{a}{C\rho_N-L}\right)^s\right).$$ Then we have
$$\bar{A}^+_{\rho_N}\geq \one_{\fA_{\rho_N}},\,
\EXP\bar{A}^+_{\rho_N}-\mu(\fA_{\rho_N})\leq C\rho_N^{-(s+1)}\text{ and }\left\|\bar{A}^+_{\rho_N}\right\|_{BV}\leq 2\mu(\fA_{\rho_N})+4.$$ For $j=1, \dots, r,$ define
$$\hA_j=\begin{cases} \bar{A}_{\fr_N}^+, & \text{if } j\neq \brj; \\  \bar{A}_{\fr_N}^+ \left(\bar{A}_{\fr_N}^+ \circ f^p\right), &
\text{if } j=\brj. \end{cases} $$
Let $\DS \Lambda=\max_{x\in\mathbb{S}}|f'(x)|.$ Then, we have
\Bea &&\mu\left(\prod_{j=1}^{r+1} \fA_{\fr_N}^{k_j} \right) \leq \EXP\left(\prod_{j=1}^r \hA_j \circ f^{k_j}\right)\\
&\leq&\left(\prod_{j\neq\brj} \left(\EXP\hA_j +C\left\|\bar{A}_{\fr_N}^+\right\|_{BV}\theta^{K\ln N}\right)\right) \left(\EXP\hA_\brj +C\left\|\bar{A}_{\fr_N}^+\right\|_{BV}^2\Lambda^p\theta^{K\ln N}\right)\leq \prod_{j=1}^r \left(\EXP\hA_j+N^{-10}\right)
\Eea
if $s(n)=K\ln N$ and $K$ is large enough.
Noting that
$\DS
\EXP\hA_\brj=\EXP\left(\bar{A}_{\fr_N}^+ (\bar{A}_{\fr_N}^+ \circ f^p)\right)\leq C \sigma_{\fr_N} \theta^p
$
by \eqref{PWM2}, we conclude that each term for $p$ as in the case (ii) is at most $C\sigma_{\fr_N}^r \theta^p.$
Summing over all $r$ tuples and over all $p\in [p_0, \varepsilon\ln N],$ we obtain a total contribution of $C N^r \sigma_{\rho_N}^r \theta^{p_0}$ for some constant $C$ independent of $p_0.$

(iii) The terms where $\eps\ln N< p\leq K \ln N$ contribute $O(\sigma_{\rho_N}^{r+ \eta})$ for some $\eta=\eta(\eps)>0.$ Therefore, the total contribution of those terms is $O\left(N^r (\ln N) \; \sigma_{\rho_N}^{r+\eta}\right).$

(iv) The terms where $p>K\ln N$ for sufficiently large $K$ contribute $O(\sigma_{\rho_N}^{r+1}),$ similar to the proof
of $(M1)_r.$ Hence, the total contribution of these terms is $O\left((N \sigma_{\rho_N})^{r+1}\right).$

Summing the estimates from cases (i)--(iv), we observe that $\bar\tau_{r+1}^N$ is negligible in comparison to $\tau_{*}^N.$ This completes the proof of $(M4)_{I_1,\dots, I_r; i_1, \dots, i_r}$ for non-periodic points.

Next, we consider the case where the point $x$ is periodic. Suppose the period of $x$ is $q,$ i.e., $f^q(x)=x.$ We use the notation $\tilde\fA_{\rho_N,j}$ as defined in~\eqref{tlfAroj} and define
$$\tilde{A}_{I_1,\dots, I_r; i_1, \dots i_r}^N=\left\{\omega\in\Omega:\,\exists\, k_j\,\,\text{such that}\,\,k_j/N\in  I_j\,\,\text{and }\,\,\omega\in\tilde\fA^{k_j}_{\fr_N, i_j} \right\}.$$
We now proceed to prove $(M4)_{I_1, \dots I_r; i_1, \dots, i_r}$ for $\tilde{A}_{I_1,\dots, I_r; i_1, \dots i_r}^N.$

Observe that $$f^{-p_0}\fA_{\fr_N, j} = \left(f^{-(p_0 - q)}\fA_{\fr_N, j} \mathrel{\scalebox{1}{$\bigcap$}} f^{-p_0}\fA_{\fr_N, j}\right) \mathrel{\scalebox{1}{$\bigcup$}} \tilde\fA_{\rho_N,j}.$$
Indeed if $f^k y \in \fA_{\fr_N, j},$  $f^n y \in \fA_{\fr_N, j}$ can only occur if $n=k + iq$ for  $i \in \mathbb{Z}.$ Therefore, $(p_0-k)$ must be a multiple of $q$. In this case, $f^{p_0} y \in \fA_{\fr_N, j}$ implies $f^{p_0 - q} y \in \fA_{\fr_N, j}$, since $f$ is expanding.

In addition, we have
\bea\label{nint}f^{-(p_0-q)}\fA_{\fr_N, j}\mathrel{\scalebox{1}{$\bigcap$}}f^{-p_0}\fA_{\fr_N, j}=\emptyset\eea
whenever
$0<\eps<\left(\gamma^{q/s}-1\right)c_j/\left(\gamma^{q/s}+1\right),$ where $\eps$ is from the definition of
$$\fA_{\rho_N,j}=\left\{y:\frac{\phi(y)}{\rho_N} \in [c_j-\eps, c_j+\eps]\right\},$$
and $\DS \gamma=\min_{x\in\mathbb{S}}|f'(x)|>1.$ If $f^{p_0-q}(y)\in \fA_{\fr_N, j},$ we have
$$\left|f^{p_0}(y)-x\right|\geq\gamma^q\left|f^{p_0-q}(y)-x\right|\geq\gamma^q\left(\frac{a}{\rho_N(c_j+\varepsilon)+L}\right)^s
>\left(\frac{a}{\rho_N(c_j-\varepsilon)-L}\right)^s$$
where $\DS L=\max_{x\in\mathbb{S}}\left|\psi(x)\right|.$
Therefore, we conclude that $f^{p_0}(y)\notin \fA_{\fr_N, j},$ and \eqref{nint} follow.

\noindent
Hence
$\DS \tilde\fA_{\rho_N,j}=f^{-p_0}\fA_{\fr_N, j}$ for $0<\eps<\frac{(\gamma^{q/s}-1)c_j}{\gamma^{q/s}+1}$
and $\DS \mu\left(\tilde\fA_{\rho_N,j}\right)=\mu\left(f^{-p_0}\fA_{\fr_N, j}\right)=\sigma_{\rho_N,j}.$

We will show that $$\lim_{N\rightarrow\infty}\frac{\mu\left(\tilde{A}_{I_1,\dots, I_r; i_1, \dots i_r}^N\right)}{N^r\Pi_{j=1}^r\left(\sigma_{\rho_N,i_j}|I_j|\right)}=1.$$

Note that
$$\sum_{k_j/N\in I_j}\mu\left(\bigcap_{j=1}^r \tilde\fA^{k_j}_{\fr_N, i_j} \right)=\sum_{k_j/N\in I_j}\mu\left(\bigcap_{j=1}^r \fA^{k_j}_{\fr_N, i_j} \right)=\bar\tau_{*}^N,$$
which coincides with the non-periodic case. We prove $(M4)_{I_1, \dots, I_r; i_1, \dots, i_r}$ using the inequality:
\begin{equation}
\label{M4*Lower}
 \mu\left(\tilde{A}_{I_1,\dots, I_r; i_1, \dots i_r}\right)\geq\bar\tau_{*}^N-\sum_{t=r+1}^{2r}\tilde\tau_t^N,
\end{equation}
where $$\tilde\tau_t^N=\sum_{k_j/N\in I_j}\mu\left(\tilde\fA^{k_1}_{\fr_N,i_1}\cap \dots \cap \tilde\fA^{k_r}_{\fr_N,i_r} \cap
\tilde\fA^{k_{r+1}}_{\fr_N,i_{r+1}}\cap\cdots\cap\tilde\fA^{k_t}_{\fr_N,i_t}\right).$$

We now show that
$\DS \frac{\tilde\tau_{r+1}^N}{\bar\tau_{*}^N}\to 0\,\,\,\text{as}\,\,\,N\to\infty.$
The proof  for  other values of $t$ follows similarly.  We consider several cases depending on the value of
$\DS p=\min_{j=1,\dots, r} |k_{r+1}-k_j|$.

(1) If $p<p_0$ where $p_0$ is specified in the definition of $\tilde\fA_{\fr_N,j},$ we take $j_0$ such that  $k_{r+1}=k_{j_0}+p$~(the case where $k_{r+1}=k_{j_0}-p$ is similar). Since $d(I_i,I_j)>\eps$ for $i\neq j,$
both $k_{r+1}/N$ and $k_{j_0}/N$ belong to the same component $I_j,$ thus $i_{r+1}=i_{j_0}.$ We then obtain
$$
\tilde\fA^{k_{j_0}}_{\fr_N,i_{j_0}}\cap\tilde\fA^{k_{r+1}}_{\fr_N,i_{r+1}}\subset f^{-k_{j_0}}\left(
f^{-p_0}\fA_{\fr_N,i_{j_0}}\mathrel{\scalebox{1}{$\bigcap$}}\left(f^{-p}\fA_{\fr_N,i_{j_0}}^c\cap\cdots\cap
f^{-p-p_0+1}\fA_{\fr_N,i_{j_0}}^c\right)\right)=\emptyset.
$$

(2) If $p\geq p_0,$ we take $C>0$ such that
$$\fA_{\rho_N,i_j}\subset\fA_{\rho_N}=\{y:\phi(y)>C\rho_N\}$$ for all $i_j,$ $1\leq j\leq r.$
It follows that $$\tilde\fA_{\fr_N,i_j}\subset f^{-p_0}\fA_{\fr_N,i_j}\subset f^{-p_0}\fA_{\rho_N}$$
and hence
$$\mu\left(\mathrel{\scalebox{1}{$\bigcap$}}_{j=1}^{r+1}\tilde\fA^{k_j}_{\fr_N,i_j}\right)\leq
\mu\left(\mathrel{\scalebox{1}{$\bigcap$}}_{j=1}^{r+1} f^{-p_0}\fA^{k_j}_{\fr_N}\right)
=\mu\left(\mathrel{\scalebox{1}{$\bigcap$}}_{j=1}^{r+1} \fA^{k_j}_{\fr_N}\right).$$
Thus, by the same arguments as (ii)-(iv) in the non-periodic case, we conclude that
$\tilde\tau_{r+1}^N$ is negligible compared to $\bar\tau_{*}^N.$ This completes the proof of $(M4)_{I_1,\dots, I_r; i_1, \dots, i_r}$ for periodic points.
\end{proof}

\begin{proof}[Proof of Lemma \ref{LmTwoHumps}]
Let $$H_{N, \breps}=\left\{y:\exists\,\brn,\,\hn\leq N,\,\,(\ln N)^{(\alpha-\bralpha)/2} \leq  |\brn-\hn|\leq 2 s(N),\,\, \min\{X_\brn, X_\hn\}>\breps N^{1/s}(\ln N)^{\alpha-1}\right\},$$
and
$$\tilde{H}_{k, \breps}=\left\{y:\exists\,\brn,\,\hn\leq 2^{k+1},\,\,k^{(\alpha-\bralpha)/2} \leq  |\brn-\hn|\leq 2 s(2^{k+1}),\,\, \min\{X_\brn, X_\hn\}>\breps 2^{k/s}k^{\alpha-1}\right\}.$$
Since $H_{N, \breps}\subset\tilde{H}_{k, \breps}$ for $2^k\leq N\leq 2^{k+1},$
it suffices to show  that $\tilde{H}_{k, \breps}$ occurs only for finitely many $k.$ Using \eqref{PWM2} and similar method to the proof of Proposition \ref{AdTag} we conclude that
$$ \mu(\tilde{H}_{k, \breps})\leq \frac{C2^{k+1}(k+1) }{2^k k^{s(\alpha-1)} } \theta^{k(\alpha-\bralpha)/2}=
2C k^{s(1-\alpha)+1} 2^{k(\alpha-\bralpha)\ln\theta/2}.$$
Hence $\DS \sum_k \mu(\tilde{H}_{k, \breps})<\infty$ and the result follows from the classical
Borel-Cantelli lemma.
\end{proof}

\begin{proof}[Proof of Lemma \ref{LmCollar}]
(a) Let $\DS\hat{T}_{N,j}=\sum_{n=1}^N \one_{\hB_{N,j}}\circ f^n,$ where  ${\hB_{N,j}}$ in given by \eqref{hBNj}.
Since
$\DS \frac{\brbrS_{N,j}}{N^{1/s} (\ln  N)^{(j+1)\delta}} \leq\hat{T}_{N,j},$
we have
$$
\EXP \left(\left(\frac{\bar{\bar S}_{N,j}}{N^{1/s} (\ln N)^{(j+1)\delta}}\right)^m \right)\leq\EXP\hat{T}_{N,j}^m.
$$
Next, we consider
$$\EXP\left(\binom{\hat{T}_{N,j}}{m}\right)=
\sum_{1\leq k_1<\cdots< k_m\leq N}
\mu\left(\bigcap_{j=1}^m f^{-k_j}\hB_{N,j}\right)
$$and prove that
\bea\label{bnTNj}
\sum_{1\leq k_1<\cdots< k_m\leq N}
\mu\left(\bigcap_{j=1}^m f^{-k_j}\hB_{N,j}\right)\leq C\left(N\mu\left(\hat{B}_{N,j}\right)\right)^m
\eea
for non-periodic $x.$
Since $j\leq0$ and $s<1$, we conclude that \eqref{bnTNj}  is less than
$\DS \frac{C_1}{ (\ln N)^{s j \delta m}}$. For $y\in \hB_{N,j}$ we have
$\phi(y)^m\left(N\mu\left(\hat{B}_{N,j}\right)\right)^m\leq C_2 N^{m/s}(\ln N)^{\delta m},$ because on this set
$$N^{1/s}(\ln N)^{j\delta}<\phi(y)\leq N^{1/s}(\ln N)^{(j+1)\delta}.$$
Thus combining \eqref{bnTNj}  with the identity
$$\EXP \hat{T}_{N,j}^m= \sum_{k=0}^{m} \left( \sum_{j=0}^{k} (-1)^{k-j} \binom{k}{j} k^m \right) \EXP\left(\binom{\hat{T}_{N,j}}{k}\right),$$
we obtain
$$
\EXP\brbrS_{N,j}^m\leq C N^{m/s}(\ln N)^{\delta m}.
$$
We now prove \eqref{bnTNj}. We consider several cases depending on the value of
$\DS p=\!\!\!\min_{1\leq j\le m-1}(k_{j+1}\!-\!k_j).$

(i) Since $x$ is non-periodic,  for each fixed $p_0,$ we have
$$\hB_{N,j} \cap f^{-p} \hB_{N,j} = \emptyset \,\,\,\text{for all}\,\,\, p < p_0$$
whenever  $N$ is sufficiently large because $x$ is non-periodic. Consequently, the contribution of terms corresponding to tuples with $p\leq p_0$ vanishes for large $N.$

(ii) If $p_0\leq p\leq \varepsilon\ln N,$ suppose that $\Sep_N(k_1, \dots, k_m)=l<m.$  Arguing as in case $(M4)_*$(ii) of the proof of  Proposition~\ref{AdTag}, we conclude that each term is smaller than $C \mu \left(\hat{B}_{N,j}\right)^l \theta^p$ for fixed $p.$
Since the total number of terms is $O\left(N^l p^{m-l}\right),$ it follows that the total contribution of the terms with $p=\hp$ is bounded above by
$CN^l \hp^{m-l} \mu\left(\hat{B}_{N,j}\right)^l \theta^\hp. $ Summing over tuples with  $\hp\in [p_0, \eps \ln N],$ we obtain a total contribution of
$$C N^l {p_0}^{m-l} \mu\left(\hat{B}_{N,j}\right)^l \theta^{p_0}$$
for some constant $C$ independent of $N,$ where  $p_0$ is sufficiently large.

(iii) If $\varepsilon\ln N< p\leq K\ln N$ and $\Sep_N(k_1, \dots, k_m)=l,$ then arguing similarly to case $(M4)_*$(iii)  we obtain that
$$ \mu\left(\bigcap_{j=1}^m f^{-k_j}\hB_{N,j}\right)\leq C \mu\left(\hat{B}_{N,j}\right)^{l+\eta}. $$
Summing over all $r$ tuples and over all $p\in (\varepsilon\ln N, K\ln N],$ we conclude that the total contribution is at most $C \left(N\mu\left(\hat{B}_{N,j}\right)\right)^l \mu\left(\hat{B}_{N,j}\right)^\eta\left(\ln N\right)^{m-l}.$

(iv) If $\Sep_N(k_1, \dots, k_m)=m,$ following the argument in case $(M1)_r$ of the proof of  Proposition~\ref{AdTag}, we obtain that
$$\mu\left(\bigcap_{j=1}^m f^{-k_j}\hB_{N,j}\right)\leq C \mu\left(\hat{B}_{N,j}\right)^m.$$

Summing over all tuples $(k_1, \dots, k_m)$ and noting that
$$\left(N\mu\left(\hat{B}_{N,j}\right)\right)^l\leq C \left(N\mu\left(\hat{B}_{N,j}\right)\right)^m\,\,\,\text{for}\,\,\,l < m,$$
we conclude that \eqref{bnTNj} holds, thus completing the proof  of Lemma \ref{LmCollar}(a).

(b) Suppose $x$ is a periodic point with period $q,$ i.e., $f^q(x)=x.$ If there exist $k,\,l\in\mathbb{Z}_+$ such that
$$\left\{f^ky,f^{k+q}y,\ldots,f^{k+(l-1)q}y\right\}\subset\hat{B}_{N,j}\,\,\,\text{and}\,\,\,f^{k+lq}y\notin\hat{B}_{N,j},$$
then for $p=\ln\ln N$ and sufficiently large $N,$ we have
$$f^{k+lq+i}y\notin\hat{B}_{N,j}\,\,\,\text{for all}\,\,\,i\leq p.$$
Furthermore, if there exists $n>k+lq+p$ such that $f^ny\in\hat{B}_{N,j},$ we show that the value $\phi(f^n y)$ constrains the sequence $\phi(f^{n+q} y), \ldots, \phi(f^{n+tq} y),$ provided that $f^{n+iq} y \in \hat{B}_{N,j}$ for all $0\leq i\leq t.$ Indeed, since
$$\left|f^{n+q} y - x\right|=\left|f^{n+q} y - f^qx\right|\geq\gamma^q\left|f^n y - x\right|,$$
where $\DS \gamma=\min_{x\in\mathbb{S}}|f'(x)|,$ we have
\bea\label{Defrg}
\phi(f^{n+q} y)\leq \tilde\gamma\phi(f^n y)
\eea
for some constant $\tilde\gamma$ satisfying $\gamma^{q/s}<\tilde\gamma<1.$ Consequently, we obtain
$\phi(f^{n+iq} y) \leq \tilde\gamma^{i} \phi(f^n y)$ for $1\leq i\leq t,$ which implies
$$\sum_{i=0}^{t} \phi(f^{n+iq} y) \leq \left(\sum_{i=0}^{t} \tilde\gamma^i\right) \phi(f^n y) \leq \mathsf C \phi(f^n y),
$$
where $\DS \mathsf C = \sum_{i=0}^{\infty}\tilde\gamma^i = \frac{1}{1-\tilde\gamma}.$
Now, defining
$$
\tilde{B}_{N,j} = \left\{y: fy \notin \hat{B}_{N,j}, \cdots, f^{p-1}y \notin \hat{B}_{N,j},f^py \in \hat{B}_{N,j}\right\},
$$
we have
$$
\sum_{i=0}^{tq} \phi (f^{n+i}y) \one_{\hat{B}_{N,j}} (f^{n+i}y)= \sum_{i=0}^t \phi (f^{n+iq}y) \leq
\mathsf C \phi (f^ny) = \mathsf C \phi (f^ny) \one_{\tilde{B}_{N,j}} (f^{n-p}y).
$$

Lettting
$\DS \tilde{S}_{N,j}(y) = \sum_{k=1}^{N} \phi (f^{k+p}y) \one_{\tilde{B}_{N,j}} (f^{k} y),$
and summing over the orbit segments that enter $\hat{B}_{N,j},$ we obtain
\bea\label{ConbbS}
\brbrS_{N,j}(y) \leq \mathsf C \tilde{S}_{N,j}(y) + \sum_{i=0}^{k_1} \phi (f^{k_0+iq}y)
\eea
where
$\{f^{k_0+iq}y\}_{i=0}^{k_1}$ represents the initial orbit  segment entering $\hat{B}_{N,j}.$ Noting that
$k_1 \leq \frac{2\delta s}{-\ln \tilde{\gamma}} \ln \ln N,$
where $\tilde{\gamma}$ satisfies \eqref{Defrg}, it follows that
$$\frac{\brbrS_{N,j}}{N^{1/s}(\ln N)^{(j+1)\delta}}
 \leq C \left(\sum_{k=1}^{N} I_{\tilde{B}_{N,j}} (f^k y) + \ln \ln N \right).$$

Define
$$\tilde{T}_{N,j}(y)= \sum_{k=1}^{N} \one_{\tilde{B}_{N,j}} (f^k y),$$
Then
$$
\EXP\left(\left(\frac{\brbrS_{N,j}}{N^{1/s}(\ln N)^{(j+1)\delta}}\right)^m\right)
\leq C^m \EXP\left(\left(\tilde{T}_{N,j} + \ln \ln N\right)^m\right)
=\sum_{r=0}^{m} C^m \binom{m}{r} (\ln \ln N)^{m-r} \EXP\left(\tilde{T}_{N,j}\right)^r.
$$

Next consider the expectation
\begin{align}
\mathbb{E} \left( \binom{\tilde{T}_{N,j}}{r} \right) &= \sum_{1 \leq k_1 < \cdots < k_r \leq N} \mu \left( \bigcap_{j=1}^m f^{-k_j} \tilde{B}_{N,j} \right) 
\leq \sum_{1 \leq k_1 < \cdots < k_r \leq N,\atop k_{i+1} - k_i \geq p} \mu \left( \bigcap_{j=1}^m f^{-k_j} \hat{B}_{N,j} \right), \label{tlTNj}
\end{align}
where we used  that
$$ \bigcap_{j=1}^m f^{-k_j} \hat{B}_{N,j} = \emptyset $$
unless $k_{i+1} - k_i \geq p$ for all $0 \leq i \leq r-1$. Using the same argument as  non-periodic case, the right-hand side of \eqref{tlTNj} is bounded above by
$\DS C \left( N \mu(\hat{B}_{N,j}) \right)^r. $

Now arguing as in the non-periodic case we obtain
$$ \mathbb{E}\brbrS_{N,j}^m\leq C N^{m/s} (\ln N)^{\delta m} (\ln \ln N)^m, $$
completing the proof for the periodic  $x$.
\end{proof}

\end{document}